\documentclass[a4paper,11pt]{amsart}

\usepackage[margin=1in]{geometry}
\usepackage{graphicx}
\usepackage{amsmath}
\usepackage[utf8]{inputenc}
\usepackage{amsfonts}
\usepackage{amsthm}
\usepackage{xcolor}
\usepackage{cancel}
\usepackage[pagebackref=false,pdftex,bookmarksopen=true,colorlinks=true, linkcolor=red, citecolor=green]{hyperref}

\usepackage{marginnote}

\numberwithin{equation}{section}

\newtheorem{thm}{Theorem}[section]
\newtheorem{claim}[thm]{Claim}
\newtheorem{lemma}[thm]{Lemma}
\newtheorem{prop}[thm]{Proposition}

\DeclareMathOperator\supp{supp}

\newcommand{\R}{\mathbb{R}}

\newcommand{\N}{\mathbb{N}}

\newcommand{\LL}{\mathcal{L}}

\newcommand{\DD}{\mathcal{D}}

\begin{document}
\title[Controllability KdV-B equation on the unbounded Domain]{Controllability Aspects of the Korteweg-de Vries Burgers Equation on Unbounded  Domains}
\author[F.A. Gallego]{F.A. Gallego}
\address{Institute of Mathematics, Federal University of Rio de Janeiro, UFRJ, P.O. Box 68530, CEP 21945-970, Rio de Janeiro, RJ, Brazil.}
\email{fgallego@ufrj.br, ferangares@gmail.com}
\date{}

\begin{abstract}
The aim of this work is to consider the controllability problem of the  linear system associated to Korteweg-de Vries Burgers equation posed in the whole real line. We obtain a sort of exact controllability for solutions in $L^2_{loc}(\R^2)$ by deriving an internal observability inequality and a Global Carlemann estimate. Following the ideas contained in \cite{rosier2000}, the problem is reduced to prove an approximate theorem. 
\end{abstract}

\maketitle

\section{Introduction}
The Korteweg-de Vries Burgers equation (KdV-B) was derived by Su and Gardner \cite{su1969korteweg} for a wide class of nonlinear system in the weak nonlinearity and long wavelength approximation. This equation has been obtained when including electron inertia effects in the description of weak nonlinear plasma waves \cite{hu1972}. The KdV-Burgers equation has also been used in a study of wave propagation through liquid field elastic tube \cite{johnson1970} and for a description of shallow water waves on viscous fluid. This model can be thought of as a composition of the KdV and Burgers  equation,
\begin{equation}\label{e1-5'}
u_t-\delta u_{xx}+u_{xxx}+u^pu_x=0. 
\end{equation}
\vglue 0.2 cm
The equation \eqref{e1-5'} is one of the simplest evolution equations that features nonlinearity, dissipation, and dispersion. The special case 
$\delta = 0$ and $p=1$ is the classical Korteweg-de Vries (KdV) equation which arises in modeling many practical situations involving wave propagation in nonlinear dispersive media. In spite of having many works dealing with the KdV equation in the existing literature, the same cannot be asserted to the KdV–Burgers equation. This lack of results becomes more evident when we are interested in the controllability or in the asymptotic behaviour of its solutions. However, over the last years, a considerable number of stability issues concerning the KdV-B equation have received considerable attention.  In what concerns to the boundary control and stabilization problems, we refer \cite{ balogh2000,   bona1985,  gallegopazoto2015, gao2003, jia2016, jia2012, liu2002, naumkin1991, Sakthivel2009, smaoui2008} and references therein. Nonetheless, to the best of our knowledge and already noticed in  \cite{crepeau2010}, there are few results about controllability for the KdV-B equation. Recently, M. Chen in \cite{chen2016} gets the existence of time optimal control and the null controllability of the Korteweg-de Vries-Burgers equation posed in a bounded domain under effect of a control acting locally in a subset of the domain. 
\vglue 0.2cm
\noindent{\bf Theorem A} (\cite[Theorem 1.1]{chen2016})
{\em Let $I=(0,L)$ with $L>0$ and $\omega$ be nonempty open subset of $I$. Consider the KdV-B equation with external force $v$:
\begin{equation}\label{bang}
\begin{cases}
y_t-y_{xx}+y_{xxx}+yy_x=v\chi_{\omega}  & \text{in $I \times (0,+\infty)$} \\
y(0,t)=y(L,t)=y_x(0,t)=0 & \text{in $(0,+\infty)$}, \\
y(x,0)=y_0(x) & \text{in $I$}.
\end{cases}
\end{equation}
Then, For any $y_0 \in L^2(I)\setminus \{0\}$ and any $M>0$, there exist a time $T^*>0$ and a control $v^* \in L^2(I\times (0,T^*)$ such that the solution $y$ of \eqref{bang} satisfies $y(\cdot,T^*)=0$ and $\|v^*\|_{L^2(I\times (0,T^*))}=M.$}
\vglue 0.2cm
We are interested in the exact controllability results concerning a linearized Korteweg–de Vries Burgers equation posed in a \textit{unbounded domain}. In this direction, also there is not too many results in the literature, see for instance \cite{cannarsa2004}, \cite{burgosteresa2007}, \cite{micuzuazua2001}, \cite{miller2005} and  \cite{rosier2000}. In particular, Rosier in \cite{rosier2000} studies the exact boundary controllability of the linearized KdV equation on the unbounded domain $\Omega=(0,\infty),$ where the  control problem is discussed implicitly by considering the solution set without specifying the boundary conditions, such that the exact controllability does not hold for bounded energy solutions, i.e. for solutions in $L^{\infty}(0,T ;L^2(0,\infty))$. His main result reads as follows:
\vglue 0.2cm
\noindent{\bf Theorem B} (Rosier \cite[Thm 1.3]{rosier2000})
{\em Let $T, \varepsilon, b$ be positive numbers, with $\varepsilon < T$. Let $L^2(\Omega,e^{-2bx}dx)$ denote the space of (class of) measurable functions $u : \Omega \rightarrow  \R$ such that $\int_0^{\infty} u^2(x)e^{-2bx} dx < \infty$. Let $u_0 \in L^2(\Omega)$ and $u_T \in L^2(\Omega, e^{-2bx}dx)$.
Then, there exists a function $$u \in L^2_{loc}([0,T]\times [0,\infty]) \cap C([0,\varepsilon],L^2(\Omega))\cap C([T-\varepsilon,T],L^2(\Omega, e^{-2bx}dx)) $$ fulfilling}
\begin{equation*}
\begin{cases}
u_t+u_x+u_{xxx} & =0\quad \text{in $\DD'(\Omega \times (0,T))$} \\
u|_{t=0} &=u_0,\\
u|_{t=T} &=u_T.
\end{cases}
\end{equation*}
\vglue 0.2cm
\noindent In Theorem B, $u$ is locally square integrable. Actually, for a certain function $u_0$ in $L^2(0,\infty)$ and $u_T=0$ a trajectory $u$ as above cannot be found in $L^{\infty}(0,T,L^2(0,\infty))$ (see \cite[Theorem 1.2]{rosier2000}). It means that the bad behavior of the trajectories as $x\rightarrow\infty$ is the price to be paid for getting the exact controllability in the half space $\Omega$. In the whole space, the same sort of results occurs for the heat and Schrodinger equations.

\subsection*{Main Results}
In this work we treat the linear Cauchy problem associated to  Korteweg-de Vries Burgers equation \eqref{e1-5'},
\begin{equation}\label{e1-5}
\begin{cases}
u_t-u_{xx}+u_{xxx}=0  & \text{in $\R \times \R_+$} \\
u(x,0)=u_0(x) & \text{in $\R$}.
\end{cases}
\end{equation}
Our  results have affinities with the work of Rosier \cite{rosier2000} and we adapt his ideas in our problem. Such was mentioned by Rosier, his approach  applies also to many other linear PDEs for which the characteristic hyperplanes take the form $\{t = Const.\}$, for instance, the
heat equation $u_t-\Delta u=0$ and the Schrodinger equation $i u_t +\Delta u =0$. However, it is not the unique condition to obtain the controllability. In fact, in order to obtain the desired result for our problem,  the internal controllability for the model posed in a bounded domain with some appropriated boundary condition, plays a crucial role. In our case, we consider the initial boundary condition problem:
\begin{equation*}
\begin{cases}
u_t-u_{xx}+u_{xxx}=0  & \text{in $(-L,L) \times (0,T)$}, \\
u(-L,t)=u(L,t)=u_x(L,t)=0 & \text{in $(0,T)$}, \\
u(x,0)=u_0(x) & \text{in $(-L,L)$},
\end{cases}
\end{equation*}
and we prove that its solution $u$, satisfies the observability inequality 
\begin{equation}\label{new1}
\|u\|_{L^2(0,T,L^2(-L,L))} \leq C \|u\|_{L^2(0,T,L^2(\omega))}, \quad \omega=(-l,l),\quad l<L.
\end{equation}
To prove the above inequality, we follow the same approach as in \cite{capistrano2015} and \cite{glass2008}. We use a Carleman estimate and a bootstrap argument based on the smoothing effect of the KdV-Burgers equation, see Proposition \ref{observ} below. The main result in this paper reads as:
\begin{thm}\label{main}
Let $\{S(t)\}_{t\geq 0}$ denote the continuous semigroup on $L^2(\R)$ generated by the differential operator $A=\partial_x^2-\partial_x^3$ with domain $H^3(\R)$. Let $T, \varepsilon$ positive numbers, such that $\varepsilon < \frac{T}{2}$. Let $u_0, u_T \in L^2(\R)$. Then, there exists a function 
\begin{equation*}
u \in L^2_{loc}(\R^2) \cap  C([0,\varepsilon], L^2(\R))\cap C([T-\varepsilon,T], L^2(\R))
\end{equation*}
which solves
\begin{equation}\label{linearcontrol}
\begin{cases}
u_t-u_{xx}+u_{xxx}=0 & \text{in $\DD'(\R\times (0,T))$} \\
u(x,0) = u_0(x)      &  \text{in $\R$,} \\
u(x,T) =S(T)u_T(x)   &  \text{in $\R$.}\\
\end{cases}
\end{equation}
\end{thm}

The proof of the Theorem \eqref{main} is  based in \cite{rosier2000}, it combines Fursikov–Imanuvilov’s approach \cite{fursikovimanuvilov1995} for the boundary controllability of the Burgers equation on bounded domains, which is based on a global Carleman’s estimate. In order to obtain the extension to some unbounded domain, we follow Rosay’s clever proof of Malgrange–Ehrenpreis’s theorem \cite{rosay1991}, which uses an approximation theorem. The proof of the approximation theorem is based in two technical results, namely, Proposition \ref{observ} and Lemma \ref{approlema} below. The Proposition \ref{observ} refers to the internal observaility property  \eqref{new1}.

Since the semigroup $S(\cdot)$ associated to KdV-Burgers is not a group in $L^2(\R)$, the proof of the Theorem \ref{main} does not give us the exact controllability directly in the whole space, i.e, we set that the solution $u$ of the Cauchy problem \eqref{e1-5}, satisfies $u(T)=S(T)u_T$, for any $u_0$ and $u_T$ in $L^2(\R)$. However, with some minor modifications in the proof of Theorem \ref{main} as in \cite{rosier2000}, we obtain a exact controllability result in the half-space, provided that $u_0 \in L^2(0,+\infty)$ and $u_T \in L^2((0,+\infty), e^{-2bx}dx)$ for $b\geq \frac13$, namely:
\begin{thm}\label{main2}
Let $T, \varepsilon, b$ be positive numbers, with $\varepsilon < T$ and $b \geq \frac13$.  Let $u_0 \in L^2((0,+\infty))$ and $u_T \in L^2((0,+\infty), e^{-2bx}dx)$.
Then there exists a function $$u \in L^2_{loc}([0,T]\times (0,+\infty)) \cap C([0,\varepsilon],L^2((0,+\infty)))\cap C([T-\varepsilon,T],L^2((0,+\infty), e^{-2bx}dx)) $$ fulfilling
\begin{equation*}
\begin{cases}
u_t-u_{xx}+u_{xxx} & =0\quad \text{in $\DD'((0,+\infty) \times (0,T))$} \\
u|_{t=0} &=u_0,\\
u|_{t=T} &=u_T.
\end{cases}
\end{equation*}
\end{thm}

\vglue 0.4cm

The paper is outlined as follows:
\vglue 0.4cm
 - In section \ref{internatlcontrollability}, we present an internal observability inequality for some appropriate initial value problem of the KdV-B equation posed on a finite interval, which will be used to prove an approximation theorem in order to obtain our main result. 
 \vglue 0.2cm
 - Section \ref{proofmainresult} is devoted to the proof of the Theorem \ref{main}. 
 \vglue 0.2cm
 - Section \ref{furthercoments} contains some further comments and related problems of controllability. 
   \vglue 0.2cm
 - Finally,  in the Appendix \ref{appendix}, we establish some  auxiliary results.

\section{Internal Observability}\label{internatlcontrollability}

In this section, we follow the same approach as in \cite{capistrano2015} to prove a observability inequality for the linear KdV-Burgers equation posed in a bounded domain. Consider the differential operator
\begin{equation}\label{A}
A=:\partial_{xx}-\partial_{xxx}, \quad D(A):= \left\lbrace u \in  H^3(-L,L): u(-L)=u(L)=u_x(L)=0 \right\rbrace.
\end{equation}
\begin{prop}
The operator $A$ and its adjoin $A^*$ are dissipative in $L^2(-L,L)$.
\end{prop}
\begin{proof}
It is easy to see that $A^*$ is given by 
\begin{equation}\label{A*}
A^*=\partial_{xx}+\partial_{xxx}, \quad D(A^*):= \left\lbrace \varphi \in  H^3(-L,L): \varphi(-L)=\varphi(L)=\varphi_x(-L)=0 \right\rbrace.
\end{equation}
Let $u \in D(A)$, hence, 
\begin{align*}
(Au,u)_{L^2} &= \int_{-L}^{L}u_{xx}udx-\int_{-L}^{L}u_{xxx}udx = -\int_{-L}^{L}u_{x}^2dx+\int_{-L}^{L}u_{xx}u_{x}dx  \\
&= -\int_{-L}^{L}u_{x}^2dx -\frac{1}{2}u_x^2(-L)\leq 0. 
\end{align*}
Then, $A$ is a dissipative operator in $L^2(-L,L)$. Analogously, $A^*$ is also a dissipative operator in $L^2(-L,L)$.
\end{proof}

The above Proposition together with the density property of the domains $D(A)$ and $D(A^*)$ in $L^2(-L,L)$ and the closeness of the operator $A$ ($A=A^{**}$), allow us to conclude that $A$ generates a $C_0$ semigroup of contractions $\{S_L(t)\}_{t \geq 0}$ on $L^2(-L,L)$ (See \cite{pazy}) which be denoted by $S_L(\cdot)$. 
Classical existence results then give us the global well-posedness in the space $L^2(-L,L)$.
\begin{thm}\label{homog} 
Let  $u_0\in L^2(-L,L)$ and consider the initial boundary value problem 
\begin{equation}\label{ibvp}
\begin{cases}
u_t =Au & \text{in $(0,T)\times (-L,L)$,} \\
u(0,x)=u_0(x) & \text{in $(-L,L)$.}
\end{cases}
\end{equation}
Then, there exists a unique (weak) solution $u=S_L(\cdot)u_0$ of \eqref{ibvp} such that
\begin{equation*}
u \in C\left([0,T];L^2(-L,L)\right)\cap H^1\left(0,T; (H^{-2}(0,L))^2\right).
\end{equation*}
Moreover, if $u_0 \in D(A)$ then \eqref{ibvp} has a  unique (classical) solution $u$ such that
\begin{equation*}
u \in C([0,T];D(A))\cap C^1((0,T);L^2(-L,L)).
\end{equation*}
\end{thm} 
In general, the following observability inequality plays a fundamental role for the study of the controllability properties. In this case, it will be used to prove an approximation theorem stated in the next section, being a crucial key to obtain the desired result. 
\begin{prop}\label{observ}
Let $l, L, T$ be positive numbers such that $l<L$. Then there exists a
constant $C > 0$ such that, for every $u_0 \in L^2(-L, L)$, the solution of \eqref{ibvp} satisfies 
\begin{equation}\label{observ1}
\|u\|_{L^2(0,T,L^2(-L,L))} \leq C \|u\|_{L^2(0,T,L^2(\omega))}, \quad \omega=(-l,l).
\end{equation}
for some $C>0$. 
\end{prop}
The proof of the Proposition \ref{observ} was motivated by the works \cite{capistrano2015} and \cite{glass2008}. Following the methods developed in above papers, we  prove the internal observability \eqref{observ1} by using a Carleman estimative. Before to present the proof of Proposition \ref{observ}, we establish some preliminary results.
\vglue 0.4cm
\subsection*{Carleman Estimate for the KdV-Burgers equation}
In order to prove the internal observability for the KdV-Burgers equation, we  follow closely the ideas present in \cite{capistrano2015}. In such work, the authors establishes a internal Carlmenan estimate for the non-homogeneous system:
\begin{equation}\label{kdv}
\begin{cases}
q_t + q_{xxx}=f & \text{in $ (0,L) \times (0,T),$} \\
q(0,t)=q(L,t)=q_x(L,t)=0 & \text{in $(0,T),$} \\
q(x,0)=q_0(x) & \text{in $(0,L)$}.
\end{cases}
\end{equation}
where $f \in L^2(0,T;L^2(0,L))$. Note that apriori, the solution $q$ of \eqref{kdv} does not have  regularity enough to apply the Carleman estimate present in \cite[Proposition 3.1]{capistrano2015} with $f=q_{xx}$.  Hence, to get the desired Carleman estimate, we assume that $\omega = (l_1, l_2)$ with $-L < l_1 < l_2 < L$ and pick any function $\psi \in C^3([-L, L])$ with 
\begin{align}
& \text{$\psi > 0$ in $[-L,L],$}  \label{carl1}\\
& \text{$|\psi'|> 0, \psi''< 0$, and $\psi'\psi'''< 0$ in $[-L, L] \setminus  \omega,$} \label{carl2}\\
& \text{ $\psi'(-L) < 0$ and $\psi'(L) > 0,$} \label{carl3} \\
& \text{ $\min_{x \in [l_1,l_2]}\psi(x)=\psi (l_3) < \max_{x\in [l_1,l_2]}\psi(x)=\psi(l_2)=\psi(l_3), \quad \max_{x\in[-L,L]}\psi(x)=\psi(-L)=\psi(L),$} \label{carl4}\\
& \text{$\psi(-L)\leq \frac{4}{3}\psi(l_3).$} \label{carl5}
\end{align}
The existence of such a function is guaranteed in \cite{capistrano2015}. 
\begin{lemma}[\textbf{Carleman's inequality}]\label{porpcarl}
Let $T>0$. Then, there exist positive constants $s_0 = s_0(T,\omega)$ and $C = C(T,\omega)$, such that, for all $s \geq s_0$ and any $u_0 \in L^2(-L,L)$, the solution $u$ of \eqref{ibvp} fulfills  
\begin{multline}\label{car'}
\int_{0}^{T}\int_{-L}^{L}\left\lbrace \frac{s^5\psi^5}{t^5(T-t)^5}|u|^2+\frac{s^3\psi^3}{t^3(T-t)^3}|u_x|^2+\frac{s\psi}{t(T-t)}|u_{xx}|^2\right\rbrace e ^{-\frac{2s\psi(x)}{t(T-t)}}dxdt \\
+\int_0^T \left\lbrace \frac{s^3\psi(L)^3}{t^3(T-t)^3}|u_x(-L)|^2+\frac{s\psi}{t(T-t)}|u_{xx}(-L)|^2\right\rbrace e ^{-\frac{2s\psi(L)}{t(T-t)}} dt \\
\leq C \int_0^T\int_{\omega} \left\lbrace \frac{s^5\psi^5}{t^5(T-t)^5}|u|^2+\frac{s^3\psi^3}{t^3(T-t)^3}|u_x|^2+\frac{s\psi}{t(T-t)}|u_{xx}|^2\right\rbrace e ^{-\frac{2s\psi(x)}{t(T-t)}}dxdt 
\end{multline}
\end{lemma}
\begin{proof}
First, we suppose that $u_0 \in D(A)$, so that $u \in C([0, T ];D(A))\cap C^1([0, T ];L^2(-L, L))$. The general follows by a density argument.  Let $u=u(x,t)$ and  $\varphi(t,x)=\frac{\psi(x)}{t(T-t)}$, where $\psi$ is a positive function satisfying \eqref{carl1}-\eqref{carl5}. Consider 
\begin{equation*}
v := e^{-s\varphi}u \quad \text{and} \quad w := e^{-s\varphi}P(e^{s\varphi} v),
\end{equation*}
where $P$ is the differential operator given by
\begin{equation*}
P=\partial_t-\partial^2_x +\partial_x^3.
\end{equation*}
Note that 
\begin{align*}
&\partial_t(e^{s\varphi}v)=e^{s\varphi}\left\lbrace s\varphi_tv+v_t\right\rbrace, \\
&\partial_x(e^{s\varphi}v)=e^{s\varphi}\left\lbrace s\varphi_xv+v_x\right\rbrace, \\
&\partial_x^2(e^{s\varphi}v)=e^{s\varphi}\left\lbrace s\varphi_{xx}v+s^2\varphi_{x}^2v+2s\varphi_{x}v_{x}+u_{xx}\right\rbrace, \\
&\partial_x^3(e^{s\varphi}v)=e^{s\varphi}\left\lbrace s\varphi_{xxx}v+3s^2\varphi_{x}\varphi_{xx}v+3s\varphi_{xx}v_{x}+s ^3\varphi_{x}^3v+3s^2\varphi_{x}^2v_{x}+3s\varphi_{x}v_{xx}+v_{xxx}\right\rbrace.
\end{align*}
Hence, 
\begin{multline*}
P(e^{s\varphi}v)=e^{s\varphi} \left\lbrace \left( s\varphi_t+s\varphi_{xxx}+3s^2\varphi_{x}\varphi_{xx}+s^3\varphi_{x}^3-s\varphi_{xx}-s^2\varphi_{x}^2\right)v \right. \\ 
\left.+\left(3s\varphi_{xx}+3s^2\varphi_{x}^2-2s\varphi_{x}\right) v_x +\left(3s\varphi_{x}-1\right)v_{xx}+v_{xxx}+v_t\right\rbrace
\end{multline*}
and 
\begin{equation*}
w=Av+Bv_{x}+Cv_{xx}+v_{xxx}+v_{t},
\end{equation*}
Note that by definition of $w$ and using the boundary conditions of \eqref{ibvp}, we have that
\begin{equation}\label{e133'}
Av+Bv_{x}+Cv_{xx}+v_{xxx}+v_{t}=0.
\end{equation}
where
\begin{align}
A &= s(\varphi_t-\varphi_{xx}+\varphi_{xxx})+3s^2\varphi_{x}\varphi_{xx}+s^3\varphi_{x}^3-s^2\varphi_{x}^2,  \label{e134'}\\
B &= 3s\varphi_{xx}+3s^2\varphi_{x}^2-2s\varphi_{x}, \label{e135'}\\
C &= 3s\varphi_{x}-1 \label{e136'}.
\end{align}
Set $L_1v:=v_t+v_{xxx}+Bv_{x}$ and $L_2 v:=Av+Cv_{xx}$. Thus, we have
\begin{equation}\label{e4.6'}
2\int_0^T\int_{-L}^{L}L_1(v)L_2(v)dxdt\leq \int_0^T\int_{-L}^{L}\left( L_1(v)+L_2(v)\right)^2 dxdt\ = 0.
\end{equation} 
In the following, our efforts will be devoted to compute the double product in the previous equation. Let us denote by $(L_iv)_j$ the $j$-th term of $L_iv$ and $Q=[0,T]\times[-L,L]$. Then, to compute the integrals on the right hand side of (\ref{e4.6'}), we perform integration by part in $x$ or $t$:
\begin{align*}
\left( (L_1v)_1,(L_2v)_1\right)_{L^2(Q)}=-\frac{1}{2} \int_Q A_tv^2dxdt,
\end{align*}
\begin{align*}
\left( (L_1v)_2,(L_2v)_1\right)_{L^2(Q)}&=-\frac{1}{2} \int_Q A_{xxx}v^2dxdt + \frac{3}{2} \int_Q A_x v_x^2dxdt +\frac{1}{2}\int_0^T A(-L)v^2_x(-L)dt
\end{align*} 
\begin{align*}
\left( (L_1v)_3,(L_2v)_1\right)_{L^2(Q)}&=- \frac{1}{2} \int_Q (AB)_{x}v^2dxdt,
\end{align*}
\begin{align*}
\left( (L_1v)_2,(L_2v)_2\right)_{L^2(Q)}&= -\frac{1}{2} \int_Q C_{x}v^2_{xx}dxdt  +\frac{1}{2}\int_0^TC(L)v_{xx}^2(L)dt - \frac{1}{2}\int_0^TC(-L)v_{xx}^2(-L)dt
\end{align*}
\begin{align*}
\left( (L_1v)_3,(L_2v)_2\right)_{L^2(Q)}&=  -\frac{1}{2} \int_Q (BC)_{x}v^2_xdxdt -\frac{1}{2}\int_0^TB(-L)C(-L)v_{x}^2(-L)
\end{align*}
By using (\ref{e133'}), we have that
\begin{align*}
( (L_1v)_1,&(L_2v)_2)_{L^2(Q)}=-\frac{1}{2} \int_Q C \partial_t(v^2_x)dxdt  - \int_Q C_xv_xv_tdxdt   \\
=&\frac{1}{2}\int_Q C_tv^2_xdxdt  + \int_Q C_xv_x \left(Av+Bv_x +Cv_{xx}+v_{xxx}\right) dxdt\\ 
=&\frac{1}{2}\int_Q C_tv^2_xdxdt +\frac{1}{2}\int_Q AC_x (v^2)_x dxdt + \int_Q BC_xv_x^2dxdt +\frac{1}{2}\int_Q CC_x (v^2_x)_xdxdt \\
&+\int_Q C_xv_xv_{xxx}dxdt\\
=&\frac{1}{2}\int_Q C_tv^2_xdxdt -\frac{1}{2}\int_Q (AC_x)_xv^2dxdt + \int_Q BC_xv_x^2dxdt -\frac{1}{2}\int_Q (CC_x)_xv^2_xdxdt \\
&-\frac{1}{2}\int_0^T C(-L)C_x(-L)v^2_x(-L)dt +\frac{1}{2}\int_Q C_{xxx}v_x^2dxdt +\frac{1}{2}\int_0^TC_{xx}(-L)v^2_x(-L)dt \\
&-\int_QC_xv_{xx}^2dxdt -\int_0^TC_x(-L)v_x(-L)v_{xx}(-L)dt 
\end{align*}
applying Young inequality, it follows that 
\begin{align*}
( (L_1v)_1,(L_2v)_2)_{L^2(Q)}&\geq\frac{1}{2}\int_Q \left\lbrace C_t+2BC_x -(CC_x)_x+C_{xxx}\right\rbrace v^2_xdxdt-\frac{1}{2}\int_Q (AC_x)_x v^2dxdt \\
&- \int_Q C_xv_{xx}^2dxdt -\frac{1}{2}\int_0^TC_x^2(-L)v_x^2(-L) -\frac{1}{2}\int_0^Tv^2_{xx}(-L)dt \\
& +\frac{1}{2}\int_0^T\left\lbrace  C_{xx}(-L)-C(-L)C_x(-L)\right\rbrace v^2_x(-L)dt 
\end{align*}
Putting together the inequalities above, we have
\begin{align*}
2\int_0^T\int_{-L}^{L}&L_1(v)L_2(v)dxdt \geq  -\int_Q\left\lbrace A_t+A_{xxx}+(AB)_x+(AC_x)_x\right\rbrace v^2dxdt \\
&+ \int_Q \left\lbrace 3A_x - (BC)_x+C_t+2BC_x -(CC_x)_x+C_{xxx} \right\rbrace v^2_xdxdt \\
&- 3\int_Q C_xv_{xx}^2dxdt -\int_0^T(C(L)+1)v_{xx}^2(L)dt \\
&+\int_0^T\left\lbrace A(-L)-B(-L)C(-L)-C(-L)C_x(-L)+C_{xx}(-L)-C^2_{xx}(-L)\right\rbrace v_x^2(-L)dt.
\end{align*}
From \eqref{e4.6'}, it follows that
\begin{equation}\label{e137'}
\int_Q \left\lbrace D v^2 +Eu_x^2+ F v_{xx}^2\right\rbrace dxdt +\int_0^T G v_x^2(-L)dt + \int_0^T H v^2_{xx}(-L)dt \leq 0
\end{equation}
with
\begin{align}
&D = - \left( A_t+A_{xxx}+(AB)_x+(C_xA)_x\right), \label{eA'}\\
&E= 3A_x +BC_x-B_xC- (CC_x)_x+C_{xxx} +C_t, \label{eE'}\\
&F=-3C_x,   \label{eF'} \\
&G= A(-L)-B(-L)C(-L)-C(-L)C_x(-L)+C_{xx}(-L)-C^2_{xx}(-L), \label{eG'} \\
&H=-C(-L)-1. \label{eH'}
\end{align}
In order to make the reading easier, the proof will be done in several steps to estimate every terms in the integral \eqref{e137'}:
\vglue 0.4cm
\noindent \textbf{Step 1:} Estimation of $\int_Q  D u^2dxdt$.
\vglue 0.2cm
First at all, note that
\begin{align*}
A_t= &s(\varphi_{tt}-\varphi_{xxt}+\varphi_{xxxt})+3s^2\varphi_{xt}\varphi_{xx}+3s^2\varphi_{x}\varphi_{xxt}+3s^3\varphi_{x}^2\varphi_{xt}-2s^2\varphi_{x}\varphi_{xt} \\
A_{xxx} = &s(\varphi_{txxx}+\varphi_{6x}-\varphi_{5x})+9s^2\varphi_{xxx}^2+12s^2\varphi_{xx}\varphi_{4x}+3s^2\varphi_{x}\varphi_{5x}+6s^3\varphi_{xx}^3+18s^3\varphi_{x}\varphi_{xx}\varphi_{xxx} \\
&+3s^3\varphi_{x}^2\varphi_{4x}-6s^2\varphi_{xx}\varphi_{xxx}-2s^2\varphi_{x}\varphi_{4x} \\
AB =& 3s^2\varphi_{xx}\varphi_{t}+3s^2\varphi_{xx}\varphi_{xxx}-3s^2\varphi_{xx}^2+9s^3\varphi_{x}\varphi_{xx}^2+12s^4\varphi_{x}^3\varphi_{xx}-12s^3\varphi_{x}^2\varphi_{xx}+3s^3\varphi_{x}^2\varphi_{t} \\
&+12s^3\varphi_{x}^2\varphi_{xxx}+3s^5\varphi_{x}^5-5s^4\varphi_{x}^4-2s^2\varphi_{x}\varphi_{t}-2s^2\varphi_{x}\varphi_{xxx}-2s^2\varphi_{x}\varphi_{xx}+2s^3\varphi_{x}^3 \\
(AB)_x = &3s^2\varphi_{xxx}\varphi_{t}+3s^2\varphi_{xx}\varphi_{xt}+3s^2\varphi_{xxx}^2+3s^2\varphi_{xx}\varphi_{4x}-8s^2\varphi_{xx}\varphi_{xxx}+9s^3\varphi_{xx}^3+24s^3\varphi_{x}\varphi_{xx}\varphi_{xxx} \\
&+36s^4\varphi_{x}^2\varphi_{xx}^2+12s^4\varphi_{x}^3\varphi_{xxx}-24s^3\varphi_{x}\varphi_{xx}^2-12s^3\varphi_{x}^2\varphi_{xxx}+6s^3\varphi_{x}\varphi_{xx}\varphi_{t} +3s^3\varphi_{x}^2\varphi_{xt} \\
&+3s^3\varphi_{x}^2\varphi_{4x}+15s^5\varphi_{x}^4\varphi_{xx}-20s^4\varphi_{x}^3\varphi_{xx}-2s^2\varphi_{xx}\varphi_{t}-2s^2\varphi_{x}\varphi_{xt} -2s^2\varphi_{x}^2\varphi_{4x} -2s^2\varphi_{xx}^2 \\
&-2s^2\varphi_{x}\varphi_{xxx}+6s^3\varphi_{x}^2\varphi_{xx} \\
AC_x = & 3s^2\varphi_{xx}\varphi_{t}+3s^2\varphi_{xx}\varphi_{xxx}-3s^2\varphi_{xx}^2+9s^3\varphi_{x}\varphi_{xx}^2+3s^4\varphi_{x}^3\varphi_{xx}-3s^3\varphi_{x}^2\varphi_{xx} \\
(AC_x)_x = & 3s^2\varphi_{xxx}\varphi_{t}+3s^2\varphi_{xx}\varphi_{xt}+3s^2\varphi_{xxx}^2+3s^2\varphi_{xx}\varphi_{4x}-6s^2\varphi_{xx}\varphi_{xxx}+9s^3\varphi_{xx}^3+18s^3\varphi_{x}\varphi_{xx}\varphi_{xxx} \\
&+9s^4\varphi_{x}^2\varphi_{xx}^2 +3s^4\varphi_{x}^3\varphi_{xxx}-6s^3\varphi_{x}\varphi_{xx}^2-3s^3\varphi_{x}^2\varphi_{xxx}.
\end{align*}
All these estimations give us that 
\begin{align}\label{D}
D &= -15 s^5\varphi_{x}^4\varphi_{xx} +D_1 
\end{align}
where 
\begin{align*}
- D_1 = & s\varphi_{tt}+s\varphi_{xxxt}-s\varphi_{xxt}+9s^2\varphi_{xt}\varphi_{xx}+3s^2\varphi_{x}\varphi_{xxt}+3s^3\varphi_{x}^2\varphi_{xt}-4s^2\varphi_{x}\varphi_{xt} \\
&+s\varphi_{txx}+s\varphi_{6x}-s\varphi_{5x}+18s^2\varphi_{xxx}^2+15s^2\varphi_{xx}\varphi_{4x}+3s^2\varphi_{x}\varphi_{5x}+24s^3\varphi_{xx}^3+60s^3\varphi_{x}\varphi_{xx}\varphi_{xxx} \\
&+6s^3\varphi_{x}^2\varphi_{4x}-20s^2\varphi_{xx}\varphi_{xxx}-4s^2\varphi_{x}\varphi_{4x}+6s^2\varphi_{xxx}\varphi_{t} +45s^4\varphi_{x}^2\varphi_{xx}^2+15s^4\varphi_{x}^3\varphi_{xxx} \\
&-30s^3\varphi_{x}\varphi_{xx}^2 -15s^3\varphi_{x}^2\varphi_{xxx}+6s^3\varphi_{x}\varphi_{xx}\varphi_{t}+3s^3\varphi_{x}^2\varphi_{xt} -20s^4\varphi_{x}^3\varphi_{xx}-2s^2\varphi_{xx}\varphi_{t}-2s^2\varphi_{xx}^2\\
& -2s^2\varphi_{x}\varphi_{xxx} +6s^3\varphi_{x}^2\varphi_{xx}.
\end{align*}
Note that \eqref{carl1}-\eqref{carl5} imply that 
\begin{equation}\label{esti}
|\varphi_{t}| \leq K_1 \varphi^2, \quad |\varphi_{tt}| \leq K_2 \varphi^3, \quad \text{and} \quad |\partial^k_x\varphi| \leq C_k \varphi,
\end{equation}
where $K_1, K_2$ and $C_k$ are positive constants depending of $L$, $\omega$ and $k$. Therefore, there exist a constant $k_1>0$, such that
\begin{equation*}
|D_1|\leq k_1 s^4\varphi^4, \quad (x,t) \in (-L,L)\times (0,T)
\end{equation*}
and 
\begin{equation*}
|15 s^5\varphi_{x}^4\varphi_{xx}|\leq k_1 s^5\varphi^5, \quad (x,t) \in \omega\times (0,T).
\end{equation*}
We infer from \eqref{carl2} that for some $k_2 > 0$,
\begin{equation*}
-15 s^5\varphi_{x}^4\varphi_{xx} = -15 s^5 \frac{(\psi')^4\psi''}{t^5(T-t)^5} \geq k_2 s^5 \varphi^5, \quad (x,t) \in ([-L,L]\setminus\omega)\times (0,T). 
\end{equation*}
Taking \eqref{D} into a count and using the above estimates in the first integral in \eqref{e137'}, we obtain
\begin{align*}
\int_Q D v^2 dxdt  =&  -15 \int_Q  s^5\varphi_{x}^4\varphi_{xx} v^2 dxdt + \int_Q D_1 v^2 dxdt \\
\geq & k_2\int_{(0,T)\times([-L,L]\setminus\omega)}(s\varphi)^5 v^2 dxdt - k_1\int_{(0,T)\times\omega}(s\varphi)^5 v^2 dxdt - k_1\int_{Q}(s\varphi)^4 v^2 dxdt \\
=& \int_{Q}\left\lbrace k_2(s\varphi)^5-k_1(s\varphi)^4\right\rbrace v^2 dxdt - (k_1+k_2)\int_{(0,T)\times\omega}(s\varphi)^5 v^2 dxdt.
\end{align*}
Thus, there exist a positive constants $C_1$ and $C_2$, such that for any $s\geq s_1$ with $s_1$ large enough, we obtain
\begin{align}\label{cD'}
\int_Q D v^2 dxdt  \geq & C_1\int_{Q}(s\varphi)^5 v^2 dxdt - C_2\int_{(0,T)\times\omega}(s\varphi)^5 v^2 dxdt.
\end{align}
\vglue 0.4cm
\noindent \textbf{Step 2:} Estimation for  $\int_Q  E v^2_xdxdt$ and $\int_Q  F v^2_{xx}dxdt$.
\vglue 0.2cm
Note that
\begin{align*}
BC_x = & 9s^2\varphi_{xx}^2+9s^3\varphi_{x}^2\varphi_{xx}-6s^2\varphi_{x}\varphi_{xx}, \\
-B_xC= & -27s^2\varphi_{x}\varphi_{xxx}-18s^3\varphi_{x}^2\varphi_{xx}+12s^2\varphi_{x}\varphi_{xx}+9s\varphi_{xxx}-2s\varphi_{xx}, \\
3A_x = &3s\varphi_{xt}+3s\varphi_{4x}-3s\varphi_{xxx}+9s^2\varphi_{xx}^2+9s^2\varphi_{x}\varphi_{xxx}+9s^3\varphi_{x}^2\varphi_{xx}-6s^2\varphi_{x}\varphi_{xx},\\
CC_x = &9s^2\varphi_{x}\varphi_{xx}-3s\varphi_{xx}, \\
-(CC_x)_x = &-9s^2\varphi_{xx}^2-9s^2\varphi_{x}\varphi_{xxx}+3s\varphi_{xxx}, \\
C_{xxx}+C_t = &3s\varphi_{4x}+3s\varphi_{xt}.
\end{align*}
Putting together these expressions, we have
\begin{align*}
E = 6s\varphi_{xt}+6s\varphi_{4x}+9s^2\varphi_{xx}^2-27s^2\varphi_{x}\varphi_{xxx}+9s\varphi_{xxx}-2s\varphi_{xx}.
\end{align*}
We infer from \eqref{carl1}-\eqref{carl5} that for some $k_3 > 0$ and $k_4>0$, 
\begin{gather*}
9s^2\varphi_{xx}^2-27s^2\varphi_{x}\varphi_{xxx} =\frac{9s^2\left((\psi'')^2-3\psi'\psi'''\right)}{t^2(T-t)^2}\geq k_3 (s\varphi)^2, \quad (x,t) \in \left([-L,L]\setminus \omega\right)\times (0,T), \\
\left|9s^2\varphi_{xx}^2-27s^2\varphi_{x}\varphi_{xxx}\right| \leq  k_4 (s\varphi)^2, \quad (x,t) \in  \omega\times (0,T), \\
\left| 6s\varphi_{xt}+6s\varphi_{4x}+9s\varphi_{xxx}-2s\varphi_{xx}\right| \leq k_4 s\varphi^2, \quad (x,t) \in  [-L,L]\times (0,T). 
\end{gather*}
By using the above estimates, we obtain
\begin{align*}
\int_Q E v^2_x dxdt  =&  \int_Q \left(9s^2\varphi_{xx}^2-27s^2\varphi_{x}\varphi_{xxx}\right) v^2_x dxdt \\
&+ \int_Q \left(6s\varphi_{xt}+6s\varphi_{4x}+9s\varphi_{xxx}-2s\varphi_{xx} \right) v^2_x dxdt \\
\geq & k_3\int_{(0,T)\times([-L,L]\setminus\omega)}(s\varphi)^2 v^2_x dxdt - k_4\int_{(0,T)\times\omega}(s\varphi)^2 v^2_x dxdt - k_4\int_{Q}s\varphi^2 v^2_x dxdt \\
=& \int_{Q}\left\lbrace k_3(s\varphi)^2-k_4s\varphi^2\right\rbrace v^2_x dxdt - (k_3+k_4)\int_{(0,T)\times\omega}(s\varphi)^2 v^2_x dxdt.
\end{align*}
Thus, there exist positive constants $C_2$ and $C_3$, such that, for any $s\geq s_2$ with $s_2$ large enough, we obtain
\begin{equation}\label{cE'}
\int_Q E v^2_xdxdt \geq C_2\int_{Q} (s\varphi)^2v^2_x dxdt - C_3\int_{(0,T)\times\omega}(s\varphi)^3 v^2_x dxdt.
\end{equation}
Moreover, note that \eqref{carl2} implies that there exist $C_4>0$ and $C_5>0$ such that
\begin{align*}
F= -9s\varphi_{xx} = - \frac{9s\psi''}{t(T-t)} \geq C_4 s\varphi, \quad (x,t) \in \left([-L,L]\setminus \omega\right)\times (0,T)
\end{align*}
and 
\begin{align*}
|9s\varphi_{xx}| \leq C_5 s\varphi, \quad (x,t) \in \left(\omega\right)\times (0,T).
\end{align*}
Furthermore, 
\begin{equation}\label{cF'}
\int_Q F v^2_{xx}dxdt \geq C_4\int_{Q} s\varphi v^2_{xx} dxdt - (C_4+C_5)\int_{(0,T)\times\omega}s\varphi v^2_{xx} dxdt.
\end{equation}
\vglue 0.4cm
\noindent \textbf{Step 3:} Estimation for  $\int_0^T G u_x^2(L)dt$ and $\int_0^T H u^2_{xx}(L)dt$:
\vglue 0.2cm
First, observe that
\begin{align*}
A = & s\varphi_t-s\varphi_{xx}+s\varphi_{xxx}+3s^2\varphi_{x}\varphi_{xx}+s^3\varphi_{x}^3-s^2\varphi_{x}^2,  \\
-BC= & -9s^2\varphi_x\varphi_{xx} +3s\varphi_{xx}+9s^2\varphi_x^2-9s^3\varphi_x^3-2s\varphi_x, \\
-CC_x = &-9s^2\varphi_{x}\varphi_{xx}+3s\varphi_{xx}, \\
C_{xx}=&3s\varphi_{xxx}, \\
-C_x^2=&-9s^2\varphi_{xx}^2.
\end{align*}
The above identities imply that
\begin{equation*}
G= -8s^3\varphi_x^3(-L) +G_1 
\end{equation*}
where
\begin{multline*}
G_1= s\varphi_t(-L)+5s\varphi_{xx}(-L)+4s\varphi_{xxx}(-L)-15s^2\varphi_{x}(-L)\varphi_{xx}(-L) +8s^2\varphi_x^2(-L) \\
-2s\varphi_x(-L)-9s^2\varphi_{xx}^2(-L). 
\end{multline*}
We infer from \eqref{carl1}-\eqref{carl5} that for some $k_5 > 0$ and $k_6>0$, 
\begin{equation*}
-8s^3\varphi_x^3(-L)=-\frac{8s^3(\psi')^3(-L)}{t^3(T-t)^3} \geq k_5(s\varphi(L))^3, \quad t \in (0,T).
\end{equation*}
and 
\begin{equation*}
|G_1| \leq k_6(s\varphi(L))^2, \quad t \in (0,T).
\end{equation*}
Then, it follows that
\begin{equation*}
\int_0^T G v^2_x(-L)dt \geq \int_0^T \left( k_5(s\varphi(L))^3-k_5(s\varphi(L))^2\right) v^2_x(-L) dxdt.
\end{equation*}
Thus, there exists a positive constant $C_6$, such that, for any $s\geq s_3$ with $s_3$ large enough, we obtain
\begin{equation}\label{eG1'}
\int_0^T G v^2_x(-L)dt \geq C_6\int_0^T (s\varphi(L))^3v^2_x(-L) dxdt.
\end{equation}
Moreover, note that \eqref{carl3} implies that there exist $k_7>0$ such that
\begin{align*}
H=-3s\varphi_x(-L) -2 \geq k_7 s\varphi(L)-2, \quad t \in (0,T).
\end{align*}
Furthermore, 
\begin{align*}
\int_0^T H v^2_{xx}(-L)dt \geq \int_0^T \left( k_7 s\varphi(L)-2\right) v^2_{xx}(-L) dxdt.
\end{align*}
Hence, there exists a positive constant $C_7$, such that for any $s\geq s_4$ with $s_4$ large enough, we obtain
\begin{equation}\label{eH1'}
\int_0^T H v^2_{xx}(-L)dt \geq C_7\int_0^T s\varphi(L) v^2_{xx}(-L) dxdt.
\end{equation}
Combining (\ref{e137'}) together with (\ref{cD'})-\eqref{eH1'}, we obtain
\begin{multline}\label{e140}
\int_Q \left\lbrace \frac{s^5\psi^5}{t^5(T-t)^5}|v|^2+\frac{s^2\psi^2}{t^2(T-t)^2}|v_x|^2+\frac{s\psi}{t(T-t)}|v_{xx}|^2\right\rbrace dxdt \\
+\int_0^T \left\lbrace \frac{s^3\psi(-L)^3}{t^3(T-t)^3}|v_x(-L)|^2+\frac{s\psi(-L)}{t(T-t)}|v_{xx}(-L)|^2\right\rbrace dt \\
 \leq C_8  \int_{(0,T)\times \omega} \left\lbrace \frac{s^5\psi^5}{t^5(T-t)^5}|v|^2+\frac{s^3\psi^3}{t^3(T-t)^3}|v_x|^2+\frac{s\psi}{t(T-t)}|v_{xx}|^2\right\rbrace dxdt 
\end{multline}
for some $C_8 >0$. On the other hand, note that
\begin{multline*}
\int_Q \frac{s^3\psi^3}{t^3(T-t)^3}v_x^2dxdt = -\int_Q\frac{s^3\psi^3}{t^3(T-t)^3}vv_{xx}dxdt -\int_Q\frac{2s^3\psi^2\psi'}{t^3(T-t)^3}vv_{x}dxdt\\
\leq \int_Q \frac{s^5\psi^5}{2t^5(T-t)^5}|v|^2dxdt +\int_Q\frac{s\psi}{2t(T-t)}|v_{xx}|^2dxdt \\
+\max_{x \in [-L,L]}\{(\psi'(x))^2\}\int_Q \frac{s^4\psi^2}{t^4(T-t)^4}|v|^2dxdt +\int_Q\frac{s^2\psi^2}{t^2(T-t)^2}|v_{x}|^2dxdt
\end{multline*}
From \eqref{e140} and using the fact that $s$ is large enough, there exist $C>0$ such that
\begin{multline}\label{e138'}
\int_Q \left\lbrace \frac{s^5\psi^5}{t^5(T-t)^5}|v|^2+\frac{s^3\psi^3}{t^3(T-t)^3}|v_x|^2+\frac{s\psi}{t(T-t)}|v_{xx}|^2\right\rbrace dxdt \\
+\int_0^T \left\lbrace \frac{s^3\psi(L)^3}{t^3(T-t)^3}|v_x(-L)|^2+\frac{s\psi}{t(T-t)}|v_{xx}(-L)|^2\right\rbrace dt \\
 \leq C_8  \int_{(0,T)\times \omega} \left\lbrace \frac{s^5\psi^5}{t^5(T-t)^5}|v|^2+\frac{s^3\psi^3}{t^3(T-t)^3}|v_x|^2+\frac{s\psi}{t(T-t)}|v_{xx}|^2\right\rbrace dxdt. 
\end{multline}
Returning to the original variable $v=e^{-s\varphi}u$, we conclude the proof of the Lemma.
\end{proof}
In order to prove the Proposition \ref{observ}, consider following spaces
\begin{equation*}
\begin{array}{l l}
X_0:=L^2(0,T;H^{-2}(0,L)), & X_1:=L^2(0,T;H_0^{2}(0,L)), \\
\widetilde{X}_0:=L^1(0,T;H^{-1}(0,L)), & \widetilde{X}_1:=L^1(0,T;(H^{3}\cap H^{2}_0)(0,L)),
\end{array}
\end{equation*}
and
\begin{align*}
&Y_0 := L^2 ((0, T ) \times (0, L)) \cap C^0 ([0, T ] ; H^{-1} (0, L)), \\
&Y_1 := L^2 (0, T ; H^4 (0, L)) \cap C^0 ([0, T ] ; H^3 (0, L))
\end{align*}
equipped with their natural norm.  For any $\theta \in [0,1]$, we define the complex interpolation space 
\begin{align*}
X_{\theta}=\left( X_0,X_1\right)_{[\theta]}, \quad \widetilde{X}_{\theta}=\left( \widetilde{X}_0,\widetilde{X}_1\right)_{[\theta]}, \quad \text{and} \quad Y_{\theta}=\left( Y_0,Y_1\right)_{[\theta]}.
\end{align*}
For instance, we obtain that 
\begin{equation*}
\begin{array}{l l}
X_{1/2}=L^2((0,T)\times (0,L)), & \widetilde{X}_{1/2}=L^1(0,T;H_0^{1}(0,L)), \\
X_{1/4}=L^2(0,T;H^{-1} (0,L)), & \widetilde{X}_{1/4}=L^1(0,T;L^2(0,L)), 
\end{array}
\end{equation*}
and
\begin{align*}
&Y_{1/2}=L^2(0,T; H^2(0,L)) \cap C^0([0,T];H^{1}(0,L)), \\
&Y_{1/4}=L^2(0,T; H^1(0,L)) \cap C^0([0,T];L^{2}(0,L)). 
\end{align*}
We introduce the following non-homogeneous system with null initial data:
\begin{equation}\label{ibvp2}
\begin{cases}
u_t -u_{xx}+u_{xxx}=f & \text{in $ (-L,L) \times (0,T),$} \\
u(- L,t)=u(L,t)=u_x(L,t)=0 & \text{in $(0,T),$} \\
u(x,0)=0 & \text{in $(-L,L)$}.
\end{cases}
\end{equation}
\begin{lemma}\label{gain}
Let $\theta \in [1/4,1]$. If $f \in X_{\theta}\cup \widetilde{X}_{\theta}$, then the solution $u$ of \eqref{ibvp2} belongs to $Y_{\theta}$ and there exists some constant $C>0$ such that 
\begin{align*}
&\|u\|_{Y_{\theta}} \leq C \|f\|_{X_{\theta}}, \quad \text{for $f \in X_{\theta}$}\\
&\|u\|_{Y_{\theta}} \leq C \|f\|_{\widetilde{X}_{\theta}}, \quad \text{for $f \in \widetilde{X}_{\theta}$} 
\end{align*}
\end{lemma}
\begin{proof}
In order to prove the Lemma, we follow the same approach developed in \cite{glass2008}. Note that if $f \in L^2(0,T;H^{-1}(0,L))\cup L^1(0,T;L^{2}(0,L))$, then the solution $u$ of \eqref{ibvp2} belongs to $C([0,T];L^{2}(0,L))\cap L^2(0,T;H^{1}(0,L))$. Indeed,  we will suppose that $f$ belongs to $C_0^{\infty} ((0, T) \times (0, L))$ and the general case follows by density. Multiplying \eqref{ibvp2} by $u$ and integrating in $(0,t)\times (-L,L)$ with $t \in (0,T)$, we obtain that 
\begin{align*}
\frac12\int_{-L}^L u^2(t)dx +\int_0^t\int_{-L}^L u_x^2dxds \leq \int_0^t \left\langle f(s), u(s)\right\rangle_{H^{-1}\times H_0^{1}} ds, \quad \text{for $f\in L^2(0,T;H^{-1}(0,L))$}
\end{align*}
or 
\begin{align*}
\frac12\int_{-L}^L u^2(t)dx +\int_0^t\int_{-L}^L u_x^2dxds \leq \int_0^t\int_{-L}^L fudxds, \quad \text{for $f\in L^2(0,T;L^{2}(0,L))$}.
\end{align*}
Taking the supreme in $[0,T]$ and using the Young inequality, there exist a constant $C_1>0$, such that
\begin{equation}\label{g1}
\begin{cases}
\|u\|_{L^{\infty}(0,T;L^{2}(0,L))\cap L^2(0,T;H^{1}(0,L))} \leq C_1 \|f\|_{L^2(0,T;H^{-1}(0,L))}, &  \text{if $f\in L^2(0,T;H^{-1}(0,L))$} \\
\|u\|_{L^{\infty}(0,T;L^{2}(0,L))\cap L^2(0,T;H^{1}(0,L))} \leq C_1 \|f\|_{L^1(0,T;L^{2}(0,L))}, &  \text{if $f\in L^1(0,T;L^{2}(0,L))$} 
\end{cases}
\end{equation}
Now, suppose that $f \in L^2(0,T;H^{2}_0(0,L))\cup L^1(0,T;(H^{3}\cap H^{2}_0)(0,L))$,  we will prove that the solution $u$ of \eqref{ibvp2} belongs to $C([0,T];H^{3}(0,L))\cap L^2(0,T;H^{4}(0,L))$. Again, we first consider $f$ belongs to $C_0^{\infty} ((0, T) \times (0, L))$ and the general case follows by density. Consider the differential operator $$P=-\partial_x^2+\partial_x^3.$$ Let us apply the operator $P$ to the equation \eqref{ibvp2}. Thus, by using the boundary condition of the system and  the fact that $Pu=f-u_t$, it follows that 
\begin{equation*}
\begin{cases}
(Pu)_t -(Pu)_{xx}+(Pu)_{xxx}=Pf & \text{in $ (-L,L) \times (0,T),$} \\
(Pu)(- L,t)=(Pu)(L,t)=(Pu)_x(L,t)=0 & \text{in $(0,T),$} \\
(Pu)(x,0)=0 & \text{in $(-L,L)$}.
\end{cases}
\end{equation*}
Since $Pf \in L^2(0,T;H^{-1}(0,L))\cup L^1(0,T;L^{2}(0,L))$, from \eqref{g1} we infer that
\begin{equation}\label{g2}
\begin{cases}
\|Pu\|_{L^{\infty}(0,T;L^{2}(0,L))\cap L^2(0,T;H^{1}(0,L))} \leq C \|Pf\|_{L^2(0,T;H^{-1}(0,L))}, &  \text{if $Pf\in L^2(0,T;H^{-1}(0,L))$} \\
\|Pu\|_{L^{\infty}(0,T;L^{2}(0,L))\cap L^2(0,T;H^{1}(0,L))} \leq C \|Pf\|_{L^1(0,T;L^{2}(0,L))}, &  \text{if $Pf\in L^1(0,T;L^{2}(0,L))$}, 
\end{cases}
\end{equation}
for some $C>0$. Moreover, note that there exists $C_2>0$, such that
\begin{multline}\label{g3}
\|u\|_{L^{\infty}(0,T;H^3(0,L))\cap L^{2}(0,T;H^4(0,L)) } \\
\leq C_2 \left( \|u\|_{L^{\infty}(0,T;L^2(0,L))\cap L^{2}(0,T;H^1(0,L))} + \|u_{xxx}\|_{L^{\infty}(0,T;L^2(0,L))\cap L^{2}(0,T;H^1(0,L))}\right) \\
\leq C_2\left( \|u\|_{L^{\infty}(0,T;L^2(0,L))\cap L^{2}(0,T;H^1(0,L))} +\|Pu\|_{L^{\infty}(0,T;L^2(0,L))\cap L^{2}(0,T;H^1(0,L))} \right.\\
\left.+\|u_{xx}\|_{L^{\infty}(0,T;L^2(0,L))\cap L^{2}(0,T;H^1(0,L))}\right).
\end{multline}
On the other hand, 
$$L^{\infty}(0,T;H^3(0,L)) \underset{compact}{\hookrightarrow} L^{\infty}(0,T;H^2(0,L)) \hookrightarrow L^{\infty}(0,T;L^2(0,L)) $$ and 
$$ L^{2}(0,T;H^4(0,L)) \underset{compact}{\hookrightarrow}  L^{2}(0,T;H^3(0,L)) \hookrightarrow  L^{2}(0,T;H^1(0,L))$$ 
By using \cite[Lemma 8]{simon1986compact}, we have that 
\begin{multline*}
\|u_{xx}\|_{L^{\infty}(0,T;L^2(0,L))\cap L^{2}(0,T;H^1(0,L)) } \\
\leq \varepsilon \|u\|_{L^{\infty}(0,T;H^3(0,L))\cap L^{2}(0,T;H^4(0,L))} +C(\varepsilon) \|u\|_{L^{\infty}(0,T;L^2(0,L))\cap L^{2}(0,T;H^1(0,L))}.
\end{multline*}
for any $\varepsilon >0$.  Choosing an appropriate $\varepsilon>0$  from \eqref{g1}, \eqref{g2} and \eqref{g3}, we get
\begin{equation}\label{g4}
\begin{cases}
\|u\|_{L^{\infty}(0,T;H^{3}(0,L))\cap L^2(0,T;H^{4}(0,L))} \leq C_3 \|f\|_{L^2(0,T;H^{2}_0(0,L))}, &  \text{if $f\in L^2(0,T;H^{2}_0(0,L))$} \\
\|u\|_{L^{\infty}(0,T;H^{3}(0,L))\cap L^2(0,T;H^{4}(0,L))} \leq C_3 \|f\|_{L^1(0,T;H^{3}(0,L))}, &  \text{if $f\in L^1(0,T;(H^3\cap H^{2}_0)(0,L))$}. 
\end{cases}
\end{equation}
In order to complete the proof, let us define the linear map $A: f \mapsto u$. By \eqref{g1} and \eqref{g4}, $A$ continuously maps $X_{1/4}$ and $\widetilde{X}_{1/4}$ into $Y_{1/4}$, and $X_1$ and $\widetilde{X}_1$ into $Y_1$. Moreover,  the norm of the operator $A$ can be estimate as follows
\begin{equation*}
\begin{array}{l l}
\|A\|_{\LL(X_{1/4}, Y_{1/4})} \leq C_1, & \|A\|_{\LL(\widetilde{X}_{1/4}, Y_{1/4})} \leq C_1, \\
\|A\|_{\LL(X_{1}, Y_{1})} \leq C_3, & \|A\|_{\LL(\widetilde{X}_{1}, Y_{1})} \leq C_3. 
\end{array}
\end{equation*}
From classical interpolation arguments (see \cite{bergh1976}), we have that $A$ continuously maps $X_{\theta}$ and $\widetilde{X}_{\theta}$ to $Y_{\theta}$, for any $\theta \in [1/4, 1$]. Moreover, there exists a positive constant $C$, such tat the corresponding operator norms satisfy
\begin{align*}
\|A\|_{\LL(X_{\theta}, Y_{\theta})} \leq C \quad \text{and} \quad  \|A\|_{\LL(\widetilde{X}_{\theta}, Y_{\theta})} \leq C.
\end{align*}
This completes the proof.
\end{proof}
\begin{lemma}\label{lemobs}
Let $0<l<L$ and $T>0$, and $s_0$ be as in Proposition \ref{porpcarl}. Then, there exists a positive constant $C$, such that for any $s \geq s_0$ and any $u_0 \in L^2(-L,L)$, the solution $u$ of \eqref{ibvp} satisfies
\begin{equation}\label{g5}
\int_{Q} s^5 \check{\varphi}^5|u|^2 e ^{-2s \hat{\varphi}}dxdt \leq C s^{10} \int_0^T e^{s\left( 6 \hat{\varphi}-8\check{\varphi}\right)} \check{\varphi}^{31} \|u(\cdot,t)\|_{L^2(\omega)}^2 dt,
\end{equation}
where $Q=(0,T)\times(-L,L)$, $\omega=(-l,l)$, 
\begin{equation*}
\hat{\varphi}(t)=\max_{x \in [-L,L]}\frac{\psi(x)}{t(T-t)} =\frac{\psi(0)}{t(T-t)}\quad \text{and} \quad \check{\varphi}(t)=\min_{x \in [-L,L]}\frac{\psi(x)}{t(T-t)} =\frac{\psi(l_3)}{t(T-t)} 
\end{equation*}
\end{lemma}
\begin{proof}
With  Lemma \ref{gain} in hands, we can follow the same approach as in \cite{capistrano2015} and \cite{glass2008} with minor changes. In fact, by using the estimates (3.30)-(3.40) in the proof of \cite[Lemma 3.7]{capistrano2015}, we have that 
\begin{multline}\label{g6}
\int_{Q} s^5 \check{\varphi}^5|u|^2 e ^{-2s \hat{\varphi}}dxdt  \leq C s^{10} \int_0^T e^{s\left( 6 \hat{\varphi}-8\check{\varphi}\right)} \check{\varphi}^{31} \|u(\cdot,t)\|_{L^2(\omega)}^2 dt \\
+ 2\varepsilon s^{-2}\int_0^T e^{-2s \hat{\varphi}} \check{\varphi}^{-9} \|u(\cdot,t)\|_{H^{8/3}(\omega)}^2 dt,
\end{multline}
for any $\varepsilon>0$. From here, we denote by $C$, the different positive constants which may vary from place to place. Next, we will estimate adequately the integral term 
\begin{equation*}
\int_0^T e^{-2s \hat{\varphi}} \check{\varphi}^{-9} \|u(\cdot,t)\|_{H^{8/3}(\omega)}^2 dt.
\end{equation*}
This is done by a bootstrap argument based on the smoothing effect of the KdV-Burgers equation given by Lemma \ref{gain}.  Indeed, consider $u_1(x,t):=\theta_1(t)u(x,t)$, where
\begin{equation*}
\theta_1(t)=\check{\varphi}^{-1/2}e^{-s\hat{\varphi}}.
\end{equation*}
Thus, $u_1$ is the solution of
\begin{equation}\label{ibvp3}
\begin{cases}
u_{1,t} -u_{1,xx}+u_{1,xxx}=f_1 & \text{in $ (-L,L) \times (0,T),$} \\
u(- L,t)=u(L,t)=u_x(L,t)=0 & \text{in $(0,T),$} \\
u(x,0)=0 & \text{in $(-L,L)$}.
\end{cases}
\end{equation}
with $f_1(x,t)= \theta_{1,t}(t)u(x,t)$. Since $|\theta_{1,t}(t)|\leq Cs \check{\varphi}^{3/2}e^{-s\hat{\varphi}}$, we have that $f \in L^2((0,T)\times (-L,L))$ with
\begin{equation}\label{gg1-5}
\|f_1 \|^2_{L^2((0,T)\times (-L,L)}\leq Cs^2 \int_Q \check{\varphi}^{3}e^{-2s\hat{\varphi}}|u|^2dxdt,
\end{equation}
for some constant $C>0$ and all $s\geq s_0$.   Then by Lemma \ref{gain}, $u_1\in Y_{1/2}=L^2(0,T;H^2(-L,L)) \cap L^{\infty}(0,T;H^1(-L,L))$. Thus, interpolating over theses spaces, we obtain that $u_1$ belongs to space $L^4(0,T;H^{3/2}(-L,L))$ and
\begin{equation}\label{gg2-5}
\|u_1\|_{L^4(0,T;H^{3/2}(-L,L))}\leq \|f\|_{L^2((0,T)\times (0,L))}.
\end{equation}
Now, consider $u_2(x,t):=\theta_2(t)u(x,t)$, where
\begin{equation*}
\theta_2(t)=\check{\varphi}^{-5/2}e^{-s\hat{\varphi}}.
\end{equation*}
Then $u_2$ is solution of \eqref{ibvp3} with $f_2=\theta_{2,t}(t)u(x,t)$ instead $f_1$. Note that $f_2=\theta_{2,t}(t)\theta^{-1}_1(t)u_1(x,t)$, therefore since $s$ is large, from \eqref{esti} the following estimate holds $$|\theta_{2,t}(t)\theta^{-1}_1(t)|\leq Cs.$$ Thus, we have that
\begin{equation*}
\|f_2\|_{L^2(0,T; H^{1/3}(-L,L))} \leq C s \|u_1\|_{L^2(0,T; H^{1/3}(-L,L))}.
\end{equation*}
By using the embedding $H^{3/2}(-L,L)\hookrightarrow H^{1/3}(-L,L)$ and Holder inequality, it follows that  $f_2$ belongs to $X_{7/12}=L^2(0,T; H^{1/3}(-L,L))$ and, consequently,  
\begin{equation}\label{g7}
\|f_2\|^2_{L^2(0,T; H^{1/3}(-L,L))} \leq C T^{1/2} s^2 \|u_1\|^2_{L^4(0,T; H^{3/2}(-L,L))}.
\end{equation}
Thus by Lemma \ref{gain}, $u_2$ belongs to $Y_{7/12}=L^2(0,T;H^{7/3}(-L,L)) \cap L^{\infty}(0,T;H^{4/3}(-L,L))$ and 
\begin{equation}
\|u_2\|_{L^2(0,T;H^{7/3}(-L,L)) \cap L^{\infty}(0,T;H^{4/3}(-L,L))} \leq C \|f_2\|_{L^2(0,T; H^{1/3}(-L,L))}
\end{equation}\label{g8}
Finally, let $u_3(x,t):=\theta_3(t)u(x,t)$, where
\begin{equation*}
\theta_3(t)=\check{\varphi}^{-9/2}e^{-s\hat{\varphi}}.
\end{equation*}
Thus $u_3$ is solution of \eqref{ibvp3} with $f_3=\theta_{3,t}(t)u(x,t)$ instead $f_1$. Note that $f_2=\theta_{3,t}(t)\theta^{-1}_2(t)u_2(x,t)$ and $$|\theta_{3,t}(t)\theta^{-1}_2(t)|\leq Cs.$$ Thus, we have that
\begin{equation}\label{g9}
\|f_3\|_{L^2(0,T; H^{2/3}(-L,L))}  \leq C s \|u_2\|_{L^2(0,T; H^{2/3}(-L,L))}.
\end{equation}
As above, by the embedding  $H^{7/3}(-L,L)\hookrightarrow H^{2/3}(-L,L)$, we have that $f_3 \in X_{8/12}$. Thus, by Lemma \ref{gain}, $u_3$ belongs to $Y_{8/12}=L^2(0,T;H^{8/3}(-L,L)) \cap L^{\infty}(0,T;H^{5/3}(-L,L))$ with 
\begin{equation}\label{g10}
\|u_3\|_{L^2(0,T;H^{8/3}(-L,L)) \cap L^{\infty}(0,T;H^{5/3}(-L,L))} \leq C \|f_3\|_{L^2(0,T; H^{2/3}(-L,L))}.
\end{equation}
From \eqref{gg2-5}-\eqref{g10}, it yields that
\begin{equation}\label{g11}
\|u_3\|^2_{L^2(0,T;H^{8/3}(-L,L))} \leq C T^{1/2}s^4 \|f_1\|^2_{L^2((0,T)\times (-L,L))}.
\end{equation}
Then, by \eqref{gg1-5}, \eqref{g11} for $s_0$ large enough, we have  
\begin{equation*}
\int_0^T e^{-2s \hat{\varphi}} \check{\varphi}^{-9} \|u(\cdot,t)\|_{H^{8/3}(\omega)}^2 dt \leq C T^{1/2}s^2 \int_Q s^5\check{\varphi}^{5}e^{-2s\hat{\varphi}}|u|^2dxdt,
\end{equation*}
for some positive constant $C$ and for all $s \geq s_0$. Then, picking $\varepsilon = \frac{1}{4C T^{1/2}}$ in \eqref{g6}, the proof is completed.
\end{proof}

\begin{proof}[\textbf{Proof of Proposition \ref{observ}}] After the change of  variables $v(x,t)=u(T-t,L-x)$, we have
\begin{equation}\label{ibvp4}
\begin{cases}
-v_t -v_{xx}-v_{xxx}=0 & \text{in $ (0,2L) \times (0,T),$} \\
v(0,t)=v(2L,t)=v_x(0,t)=0 & \text{in $(0,T),$} \\
v(x,0)=u_0(L-x) & \text{in $(0,2L)$}.
\end{cases}
\end{equation}
Scaling in \eqref{ibvp4} by $v$ and integrating over $(0, 2L)$, it follows that
\begin{align*}
-\frac12\frac{\partial}{\partial t} \int_0^{2L} |v(t)|^2dx + \int_0^{2L} |v_x(t)|^2dx +\frac{1}{2}v_x^2(2L) = 0.
\end{align*}
Integrating over $[0,\tau]$, with $\tau \in [T/3, 2T/3]$, we get 
\begin{align*}
\|v(0)\|^2_{L^2(0,2L)}\leq \|v(\tau)\|^2_{L^2(0,2L)},
\end{align*}
Integrating again over $[T/3, 2T/3]$, the following estimate holds
\begin{equation*}
\|v(0)\|^2_{L^2(0,2L)}\leq \frac{3}{T}\int_{\frac{T}{3}}^{\frac{2T}{3}}\|v(\tau)\|^2_{L^2(0,2L)}d\tau.
\end{equation*}
Pick any $s$ large enough. Thus we obtain
\begin{align*}
\int_{\frac{T}{3}}^{\frac{2T}{3}}\|v(\tau)\|^2_{L^2(0,2L)}d\tau =\int_{\frac{T}{3}}^{\frac{2T}{3}}\|u(t)\|^2_{L^2(-L,L)}dt \leq C_1 \int_{0}^{T} \int_{-L}^{L}  s^5 \check{\varphi}^5|u|^2 e ^{-2s \hat{\varphi}}dxdt  
\end{align*}
where $C_1= \left[ \min_{t \in [T/3,2T,3]} \{ s^5\check{\varphi}^5 e ^{-2s \hat{\varphi}}\} \right]^{-1} >0.$ By Lemma \ref{lemobs}, it follows that 
\begin{align*}
\|v(0)\|^2_{L^2(0,2L)} \leq C s^{10} \int_0^T e^{s\left( 6 \hat{\varphi}-8\check{\varphi}\right)} \check{\varphi}^{31} \|u(\cdot,t)\|_{L^2(\omega)}^2 dt.
\end{align*}
Noting that $\hat{\varphi}< \frac{4}{3}\check{\varphi}$ it easy to see that the maximum of the function $\chi(t)= e^{s\left( 6 \hat{\varphi}(t)-8\check{\varphi}(t)\right)} \check{\varphi}^{31}(t)$ is attained in $T/2$ for $s$ large enough. Thus, we have that 
\begin{align}\label{g12}
\|v(0)\|^2_{L^2(0,2L)} \leq C \int_0^T  \|u(\cdot,t)\|_{L^2(\omega)}^2 dt,
\end{align}
where $C=C(s,T)>0$. Finally, by a simple change of variable in \eqref{g12}, it follows that 
\begin{align*}
\|u\|^2_{L^2((0,T)\times(-L,L))} & \leq C T \|u_0\|^2_{L^2(-L,L)} = \|v(0)\|^2_{L^2(0,2L)} \leq  C T \|u\|_{L^2( (0,T)\times \omega)}^2.
\end{align*}
This concludes the proof. 
\end{proof}
\section{Proof of the Main Result}\label{proofmainresult}
The next Proposition is carried out as in \cite[Proposition 4.1]{rosier2000} and its proof uses the internal observability \eqref{observ1} and an Approximation theorem, see the Appendix. The proof is sketched in the appendix.
\begin{prop}\label{propmain}
Let $t_1, t_2,T$ such that $0<t_1<t_2<T$ and let  $f=f(t,x)$ be any function such that 
$$f \in L^2_{loc}(\R^2) \qquad \text{and}\qquad \supp f \subset [t_1,t_2] \times \R.$$  Let $\varepsilon>0$ such that $$\varepsilon < \min (t_1,T-t_2).$$ Then, there exists $u \in L^2_{loc}(\R^2)$ such that
\begin{equation}\label{aa5}
u_t-u_{xx}+u_{xxx}=f \quad \text{in $\DD'(\R^2)$}
\end{equation}
and
\begin{equation}\label{aa6}
\supp u \subset [t_1-\varepsilon,t_2+\varepsilon]\times \R.
\end{equation}
\end{prop}

\begin{proof}[\textbf{Proof of the Main Result, Theorem \ref{main}}]
Let $u_0, u_T \in L^2(\R)$, and consider the differential operator $A=\partial_x^2-\partial_x^3$ with domain $D(A)=H^3(\R)$. It is well known that $A$ generates a $C_0$ semigroup of contraction $S(\cdot)$ on $L^2(\R)$. Thus, if $u_0, u_T \in L^2(\R)$,  $u_1(t)=S(t)u_0$ and $u_2(t)=S(t)u_T$ are the solutions of 
\begin{equation*}
\begin{cases}
\partial_tu_1-\partial_x^2 u_1 +\partial_x^3 u_1=0 & \text{in $(0,T)\times \R$} \\
u_1(0,x)=u_0(x) & \text{in $\R$}
\end{cases}
\quad \text{and} \quad 
\begin{cases}
\partial_tu_2-\partial_x^2 u_2 +\partial_x^3 u_2=0 & \text{in $(0,T)\times \R$} \\
u_2(0,x)=u_T(x) & \text{in $\R$}.
\end{cases}
\end{equation*}
Respectively. For any $\varepsilon' \in (\varepsilon, T/2)$, consider the function $\varphi \in C^{\infty}([0,T])$ given by 
\begin{equation}\label{cuttoff1}
\varphi(t) = \begin{cases}
1 & \text{if $t\leq \varepsilon'$} \\
0 & \text{if $t\geq T- \varepsilon'$}.
\end{cases}
\end{equation}
Note that the change of variable 
\begin{equation*}
u(t)=\varphi(t)u_1(t)+(1-\varphi(t))u_2(t) + w(t)
\end{equation*}
transforms \eqref{linearcontrol} in
\begin{equation}\label{linearcontrol2}
\begin{cases}
w_t-w_{xx}+w_{xxx}=\frac{d\varphi}{dt}(u_2-u_1) & \text{in $\DD'(\R\times (0,T))$} \\
w(x,0) = w(x,T) = 0 & \text{in $\R$}
\end{cases}
\end{equation}
we finish the proof by applying the Proposition \ref{propmain} written $f(t,x) = \frac{d\varphi}{dt}(t)(u_2(t,x)-u_1(t,x))$. 
\end{proof}

\section{Further comments}\label{furthercoments}

\noindent Finally, let us make some comments:
\vglue 0.2cm
1. As mentioned before, the Theorem \ref{main2} can be obtained with some minor changes of the proof of the Theorem \ref{main}. Indeed, it is easy to sse that the operators $A=\partial_x^2-\partial_x^3$ and $B=-\partial_x^2-\partial^3_x$ generate a semigroups of contraction on $L^2(0, \infty)$ and $L^2((0,\infty), e^{2bx}dx)$ for $b\geq \frac13$, respectively (for instance see \cite[Lemma 2.1]{pazoto2010}). Then, taking $u_0$ in $L^2(0,\infty)$ and $u_T$ in $L^2((0,\infty), e^{-2bx}dx)$, there exist mild solutions $u_1$ and $u_2$ of the problems,  
\begin{equation*}
\partial_t u_1 - \partial_x^2 u_1 - \partial_x^3 u_1 =0 , \quad \text{and}\quad \partial_t u_2 + \partial_x^2 u_2 +\partial_x^3 u_2 =0 , \quad \text{in $\R\times (0,T)$},
\end{equation*}
with initials data
\begin{equation*}
u_1(x,0)=\begin{cases}
u_0(x) & \text{for a.e $x >0$}, \\
0& \text{for a.e $x <0$}, 
\end{cases}, \quad \text{and}\quad 
u_2(x,0)=\begin{cases}
0& \text{for a.e $x >0$}, \\
u_T(-x) & \text{for a.e $x <0$}.
\end{cases},
\end{equation*}
With this solutions in hand, we proceed as in proof of \cite[Theorem 1.3]{rosier2000}. Thus, consider the change of function $\widetilde{u}_2(x,t)= u_2(-x,T-t)$.  Clearly,  $\partial_t \widetilde{u}_2 - \partial_x^2 \widetilde{u}_2 + \partial_x^3 \widetilde{u}_2 =0$ with $\widetilde{u}_2(x,T)=u_T(x)$ on $(0, \infty)$. In order to obtain the desired result, it is sufficient consider the  change of variable
\begin{equation*}
u(t)=\varphi(t)u_1(t)+(1-\varphi(t))\widetilde{u}_2(t)+\widetilde{w}(t), 
\end{equation*}
where $\varphi$ is the cut off function defined by \eqref{cuttoff1} and $\widetilde{w}$ is the solution of the Cauchy problem given by the Proposition \ref{propmain} with $f(x,t) = \frac{d\varphi}{dt}(t)(\widetilde{u}_2(x,t)-u_1(x,t))$. 
\vglue 0.3cm

2. The major difference with  Rosier work is the internal observability. The techniques used to prove the Proposition \ref{observ} are different from those used in the proof of the observability inequality for the KdV equation. More precisely, we developed a Carleman inequality which allow us to prove directly the observability as in \cite{capistrano2015} and \cite{glass2008}. It seems difficult to use the \textit{compactness–uniqueness argument} and \textit{the Ingham's inequality approach} used by Rosier in \cite{rosier2000,rosier}, due to the lack of $L^2(\R)$-estimates and the differential operator nature associated to KdV-Burger equation, respectively. 
\vglue 0.3cm
3. An important remark is about the approximate controllability for Pde's in $L^2(\Omega)$, when $\Omega$ is an unbounded domain. In this case, the approximate  controllability problem has a positive answer. The (simple) proof of the next Proposition can be found in the appendix of \cite{rosier2000}. 
\begin{prop}
Consider a (real) constant coefficients differential operator $Au=\sum_{i=0}^{n} a_i \frac{d^iu}{dx^i}$, with domain $D(A)=\left\lbrace u \in L ^2(\R): Au \in L^2(\R)\right\rbrace$. Assume that $n\geq 2$ (with $a_n \neq 0$) and that $A$ generates a continuous semigroup $\{S(t)\}_{t\geq 0}$ on $L^2(\R)$. Let $T>0$ and $L_1 <L_2$ be some numbers. Set 
\begin{equation*}
\mathcal{R}=\left\lbrace \int_0^TS(T-t)f(\cdot,t)dt; f \in L^2(\R^2), \supp f \subset [L_1,L_2]\times [0,T]  \right\rbrace,
\end{equation*}
where $\supp f$ denotes the support of $f$. Then $\mathcal{R}$ is a strict dense subspace of $L^2(\R)$.
\end{prop}
\vglue 0.3cm
4.  It turns out that the exact boundary controllability of the linear KdV in $L^2(0,+\infty)$ fails to be true if we restrict ourselves to solutions with bounded energy, that is, which belong to $L^{\infty}(0, T, L^2(0,+\infty))$. An implicit formulation (that is, without specification of the boundary conditions) of this fact is given in  \cite[Theorem 1.2]{rosier2000},  which shows that even the (boundary) null-controllability fails
to be true for solutions with bounded energy. This phenomenon is unknown for the linear KdV-Burgers equation \eqref{e1-5}.  Furthermore, like for the KdV equation, the nonlinear case remains a open problem. 

\section{Appendix: Proof of Proposition \ref{propmain}}\label{appendix}
In this section, we present the proof of the Proposition \ref{propmain}, it is based on an approximation theorem which has some resemblance to a result obtained in \cite[Lemma 4.4]{rosier2000}. 
\vglue 0.3cm
In order to obtain our goal in this section, we establish some results. The first of them is a global Carleman  inequality for the operator $-P^{*}=\partial_t+\partial_x^2+\partial^3_x$ proved in  \cite[Lemma 3.5]{gallegopazoto2015}.
\begin{lemma}
Let $T$ and $L$ be positive numbers. Then, there exist a smooth positive function $\psi$ on $[-L, L]$ (which depends on L) and positive constants $s_0 = s_0(L, T)$ and $C = C(L, T)$, such that, for all $s \geq s_0$ and any
\begin{equation}\label{e2-5}
q \in C^3([0,T]\times [-L,L])
\end{equation}
satisfying 
\begin{equation}\label{e3-5}
\text{$q(t, \pm L) = q_x(t,  \pm L) = q_{xx}(t, \pm L) = 0$, for $0 \leq t \leq  T$},
\end{equation}
we have that
\begin{multline}\label{car}
\int_{0}^{T}\int_{-L}^{L}\left\lbrace \frac{s^5}{t^5(T-t)^5}|q|^2+\frac{s^3}{t^3(T-t)^3}|q_x|^2+\frac{s}{t(T-t)}|q_{xx}|^2\right\rbrace e ^{-\frac{2s\psi(x)}{t(T-t)}}dxdt \\
\leq C \int_{0}^{T}\int_{-L}^{L}|q_t+q_{xx}+q_{xxx}|^2e ^{-\frac{2s\psi(x)}{t(T-t)}}dxdt.
\end{multline}
\end{lemma}
The following result uses the global Carleman inequality to obtain solutions to the KdV-Burgers equation posed, on $\R$, in the distribution sense.
\begin{prop}\label{corocarl}
Let $L>0$ and let $f=f(t,x)$ be any function such that $$ f \in L^2(\R\times (-L,-L)), \qquad \supp f \subset [t_1,t_2] \times (-L,L),$$ where $-\infty<t_1<t_2<+\infty$. Then, for any $\varepsilon>0$, there exists a positive constant $C$ and a function $v \in L^2(\R\times (-L,L))$ such that
\begin{equation}\label{e7}
v_t-v_{xx}+v_{xxx}=f \quad \text{in $\DD'(\R\times(-L,L))$},
\end{equation}
\begin{equation}\label{e8-5}
\supp v \subset [t_1-\varepsilon,t_2+\varepsilon]\times (-L,L),
\end{equation}
\begin{equation}\label{e9-5}
\|v\|_{L^2(\R\times (-L,L))} \leq C\|f\|_{L^2(\R\times (-L,L))}.
\end{equation}
\end{prop}
\begin{proof}
With the Global Carleman estimate \eqref{car} in hands, we use the same approach as in \cite[Corollary 3.2]{rosier2000}) with minor changes. 
\end{proof}

Next, we state a lemma, which may be seen as a preliminary version of the approximation Theorem \ref{aprothm} (below). Since the characteristic hyperplanes of the  linear KdV-Burgers equation  take the form $\{t = Const\}$, by using the Holmgren's uniqueness theorem, the proof of the lemma is word for word the same as the one given  for the KdV equation \cite[Lemma 4.2]{rosier2000}, hence we omit it.
\begin{lemma}\label{approlema}
Let $l_1,l_2,L,t_1, t_2, T$ be numbers, such that $0 < l_1 < l_2 < L$ and $0 < t_1 < t_2 < T$. Let $u \in L^2((0, T) \times (-l_2, l_2))$ be such that 
$$Pu = 0 \quad \text{in $(0, T) \times (-l_2, l_2)$} \quad \text{and} \quad \supp u \subset [t_1, t_2] \times (-l_2, l_2).$$ 
Then, for any $0 < 2\delta < \min(t_1, T- t_2)$ and $\eta >0$, there exist $v_1,v_2 \in L^2(-L, L)$ and $v \in L^2((0, T) \times (-L, L))$ satisfying
\begin{equation}\label{e16-5}
 Pv = 0 \quad \text{in $(0, T) \times (-L,L)$},
\end{equation}
\begin{equation} \label{e17-5}
v(t)=S_L(t-t_1+2\delta)v_1 \quad \text{for} \quad t_1-2\delta<t<t_1-\delta,
\end{equation}
\begin{equation}\label{e18-5}
v(t)=S_L(t-t_2-\delta)v_2 \quad \text{for} \quad t_2+\delta<t<t_2+2\delta,
\end{equation}
and
\begin{equation}\label{e19-5}
\|v -u\|_{L^2((t_1-2\delta,t_2+2\delta)\times(-l_1,l_1))} < \eta,
\end{equation}
where $P$ is the differential operator given by $P=\partial_t-\partial_x ^2+\partial_x^3$ and  $S_L(\cdot)$ is the $C_0$ semigroup of contraction in $L^2(-L,L)$ generated by \eqref{A}.
\end{lemma}

Now, we can establish the Approximation theorem, which differs from the approximation theorem in \cite{rosay1991} by an additional property on the support of the solution.
\begin{thm}[Approximation Theorem]\label{aprothm}
Let $n \in \N \setminus \{0,1\}$, $t_1, t_2, T$ be numbers, such that $0 < t_1 < t_2 < T$, and let $u \in L^2((0, T) \times (-n, n))$ be such that 
$$u_t-u_{xx}+u_{xxx} = 0 \quad \text{in $(0, T) \times (-n, n)$} \quad \text{and} \quad \supp u \subset [t_1, t_2] \times (-n, n).$$ Then, for any $0 < \varepsilon < \min(t_1, T- t_2)$, there exists
$v \in L^2((0, T) \times (-n - 1, n + 1))$ satisfying  
\begin{equation}\label{a1}
 v_t-v_{xx}+v_{xxx} = 0 \quad \text{in $(0, T) \times (-n - 1, n + 1)$},
\end{equation}
\begin{equation}\label{a2}
\supp v \subset [t_1 - \varepsilon, t_2 + \varepsilon]  \times (-n - 1, n + 1),
\end{equation}
and
\begin{equation}\label{a3}
\|v -u\|_{L^2((0,T )\times(-n+1,n-1))} < \varepsilon.
\end{equation}
\end{thm}
\begin{proof}
The proof combines Lemma \ref{approlema}, Proposition \ref{corocarl} and the observability inequality  \eqref{observ1} given in the Proposition \ref{observ}. With this three ingredients, the proof is obtained following the same approach used for the KdV equation \cite[Lemma 4.4]{rosier2000}. 
\end{proof}

\begin{proof}[\textbf{Proof of Proposition \ref{propmain}}]
We begin with a claim that give us a sequence of functions, which limit will be the function desired.
\begin{claim}\label{claim3}
There exist a sequence of numbers $\{t_1^n\}_{n\geq 2}$ and $\{t_2^n\}_{n\geq 2}$ such that
\begin{equation*}
t_1-\varepsilon < t_1^{n+1} < t_1^n <t_1^2< t_1 < t_2 < t^2_2< t_2^n < t_2^{n+1} < t_2 +\varepsilon, \quad \forall n \geq 2,
\end{equation*}
with
\begin{equation*}
\lim_{n\rightarrow \infty}t_1^n=t_1-\varepsilon, \qquad \lim_{n\rightarrow \infty}t_2^n=t_2+\varepsilon 
\end{equation*}
and sequence of functions $\{u_n\}_{n\geq 2}$, such that
\begin{gather}
u_n \in L^2((0,T)\times (-n,n)), \label{aa1}\\
\partial_tu_n-\partial_x^2u_n+\partial_x^3 u_n =f, \quad \text{in} \quad (0,T)\times (-n,n), \label{aa3} \\
\supp u_n \subset [t_1^n,t_2^n]\times (-n,n), \label{aa2}
\end{gather}
and, if $n >2$,
\begin{gather}\label{aa4}
\|u_n-u_{n-1}\|_{L^2((0,T)\times (-n+2,n-2))} < 2^{-n}.
\end{gather}
\end{claim}
\begin{proof}[Proof of Claim \ref{claim3}]
We will construct the sequence $\{t_i^n\}_{n\geq 2}$ and $\{u_n\}_{n\geq 2}$ by induction on $n$. Indeed, $u_2$ is given by Proposition \ref{corocarl}. Suppose that $u_2,...,u_n$ have been construct satisfying \eqref{aa1}-\eqref{aa4}. Again, applying the Proposition \ref{corocarl} with $L=n+1$, there exist $w \in L^2((0,T)\times (-n-1,n+1))$ such that 
\begin{gather*}
Pw =f, \quad \text{in} \quad \DD'((0,T)\times (-n-1,n+1)), \\
\supp w \subset [t_1^2,t_2 ^2]\times (-n-1,n+1),
\end{gather*}
where $P= \partial_t-\partial_x^2+\partial_x^3$. Note that $P(u_n-w)=0$ in $(0,T)\times (-n,n)$  and 
\begin{equation*}
\supp (u_n -w|_{(0,T)\times (-n,n)}) \subset [t_1^n,t_2^n]\times (-n,n).
\end{equation*}
From the Approximation Theorem \ref{aprothm}, there exist $v \in L^2((0,T)\times (-n-1,n+1))$ such that 
\begin{gather*}
Pw =0, \quad \text{in} \quad (0,T)\times (-n-1,n+1), \\
\supp v \subset [t_1^{n+1},t_2^{n+1}]\times (-n-1,n+1), \\
\|v-(u_n-w)\|_{L^2((0,T)\times (-n+1,n-1))} < 2^{-n-1}.
\end{gather*}
with $t_1^{n+1} < t_1^n < t_2^n < t_2^{n+1}$. Now, define $u_{n+1}= v-w$, thus \eqref{aa1}-\eqref{aa4} are fulfilled.
\end{proof}
Consider the extension
\begin{equation*}
\widetilde{u}_n = \begin{cases}
u_n & \text{in $(0,T)\times(-n,n)$} \\
0   & \text{in $\R^2 \setminus (0,T)\times(-n,n)$}.
\end{cases}
\end{equation*}
Thus, by \eqref{aa4} in Claim \ref{claim3}, $\{\widetilde{u}_n\}_{n\geq 2}$ is a Cauchy sequence in $L^2_{loc}(\R^2)$, hence there exist a function $u \in L^2_{loc}(\R^2)$, such that
\begin{equation*}
\widetilde{u}_n \longrightarrow u, \quad \text{in} \quad L^2_{loc}(\R^2).
\end{equation*}
Then \eqref{aa3} and \eqref{aa2} imply \eqref{aa5} and \eqref{aa6}, respectively.
\end{proof}
\vglue 1cm

\textbf{Acknowledgments.} The author would like to thank professor A. F. Pazoto (UFRJ) for the very useful suggestions given during the elaboration of this paper.

\nocite{*}
\bibliographystyle{plain}


\end{document}